\documentclass[12pt]{amsart}
\usepackage[margin=1in]{geometry}
\usepackage{yufei}
\numberwithin{equation}{section}

\title[]{No extremal square-free words over large alphabets}
\date{}
\author{Letong Hong \and Shengtong Zhang}
\address{Department of Mathematics, Massachusetts Institute of Technology, Cambridge, MA 02139}
\email{clhong@mit.edu, stzh1555@mit.edu}

\begin{document}
\begin{abstract}

A word is \emph{square-free} if it does not contain any \emph{square} (a word of the form $XX$), and is \emph{extremal square-free} if it cannot be extended to a new square-free word by inserting a single letter at any position. Grytczuk, Kordulewski, and Niewiadomski proved that there exist infinitely many ternary extremal square-free words. We establish that there are no extremal square-free words over any alphabet of size at least 17. 
\vspace{-2em}
\end{abstract}
\maketitle
\section{Introduction}
A \emph{word} is a finite sequence of letters over a finite alphabet. A \emph{factor} of a word is a subword of it consisting of consecutive letters. A \emph{square} is a nonempty word of the form $XX$ (examples: {\tt couscous}, {\tt "hotshots"}, {\tt "murmur"}). A word is \emph{square-free} if it does not contain a square as factor (examples: {\tt "abracadabra"}, {\tt "bonobo"}, {\tt "squares"}; non-examples: {\tt "entente"}, {\tt "referee"}, {\tt "tartar"}). It is easy to check that there are no binary square-free words of length more than 3. Thue showed in 1906 \cite{Thue} that there are arbitrarily long ternary square-free words (see \cite{Berstel}). His work is considered to be the beginning of research in combinatorics on words \cite{BP}.

Recently, Grytczuk, Kordulewski, and Niewiadomski \cite{GKN} introduced the study of \emph{extremal square-free words}.
\begin{definition}
An \emph{extension} of a finite word $W$ is a word $W' = W_1xW_2$, where $x$ is a single letter and $W_1,W_2$ are (possibly empty) words such that $W = W_1W_2$. An \emph{extremal square-free word} $W$ is a square-free word such that none of its extensions is square-free.
\end{definition}
The only binary extremal square-free words are $010$ and $101$. Via a delicate construction, Grytczuk et al. showed in  \cite{GKN} that there exist infinitely many ternary extremal square-free words. They and Pawlik also raised several open problems concerning larger alphabet sizes (\cite{GKN}, \cite{GKP}), including nonexistence of extremal square-free words over an alphabet of size 4. Mol and Rampersad \cite{MR} then classified all possible lengths of extremal ternary square-free words.
\begin{conjecture}[\cite{GKN}, \cite{MR}]\label{GKN-conj}
There exists no extremal square-free word over a finite alphabet of size at least $4$.
\end{conjecture}

To the authors' knowledge, \cref{GKN-conj} is open for any finite alphabet. Using ideas of Ter-Saakov and Zhang in \cite{TZ} and some new observations, our main result confirms their conjecture for alphabet size at least 17.

\begin{theorem}\label{main-words}
For any integer $k \geq 17$, there exists no extremal square-free word over an alphabet of size $k$.
\end{theorem}
 
Grytczuk, Kordulewski, and Niewiadomski also introduced the notion  of \emph{nonchalant words}. The sequence of nonchalant words $G_i$ is generated recursively by the following greedy procedure: $G_0$ is the empty word, and $G_{i+1}=G_i'xG_i''$ is a square-free extension of $G_i$, where $G = G_i'G_i''$ with $G_i''$
being the shortest possible suffix of $G_i$ and $x$ being the earliest possible letter such that $G_{i + 1}$ is square-free. \cref{main-words} partially affirmatively answers Conjecture 14 and 15 in \cite{GKN} for nonchalant words. More discussions can be found in \cite{GKP}. 

\begin{corollary}\label{cor-1415}
For any integer $k\ge 17$, the sequence of nonchalant words over a fixed alphabet of size $k$ converges to an infinite word.
\end{corollary}

%An \emph{abelian square} is the concatenation of two nonempty words that are permutations of each other(example: {\tt "ahooah"}), and a word is \emph{abelian square-free} if it does not contain an abelian square as factor. 

%Entringer, Jackson, and Schatz \cite{EJS} proved that every infinite binary word contains arbitrarily long abelian squares. It is clear that no ternary word of length at least 8 is abelian square-free. Erd\H{o}s asked in 1961 \cite{Erdos} whether there exist arbitrarily long abelian square-free words over a 4-letter alphabet, and Ker\"anen provided an affirmative answer in 1992 \cite{Keranen}. Over five letters, Pleasants \cite{Pleasants} showed that there exists an infinite abelian square-free word.

%One can similarly define \emph{extremal abelian square-free words} to be words that are abelian square-free but none of whose extensions is abelian square-free. A shortest such example over a 4-letter alphabet {\tt \{a,b,c,d\}} is {\tt abcdbacbdcba}. Ter-Saakov and Zhang conjectured that there exist infinitely many extremal abelian square-free words over a 4-letter alphabet \cite{TZ}. We show that if this conjecture is true, then there exists infinitely many extremal abelian square-free word over any finite alphabet of size at least 4. 

\section*{Acknowledgements}
The authors thank Prof. Joseph Gallian, Prof. Jaros\l{}aw Grytczuk, Prof. Lucas Mol, Prof. Bart\l{}omiej Pawlik, and Benjamin Przybocki for their invaluable comments to the paper.

\section{Proof of \cref{main-words}}

For a word $W$ of length $n$, we number the alphabets in $X$ from left to right as letter $1,2,\ldots, n$. We refer to the space between a letter $i$ and a letter $i + 1$ as gap $i$, and call the first and last gap $0$ and $n$. For $0 < a < b \leq n + 1$, we define the factor $W[a, b)$ as the subword of $W$ consisting of letters $a, a + 1, \ldots, b - 1$.
\begin{definition}
Let $W$ be any word. Denote $W+_b c$ as the word formed by inserting the letter $c$ at gap $b$. For positive integers $a,b$ with $a \leq b + 1$, positive integer $\ell$ and a letter $c$, we say the quadruple $(a, \ell, b, c)$ is \emph{square-completing in $W$} if the factor $(W+_b c)[a, a + \ell)$ and the factor $(W+_b c)[a + \ell, a + 2\ell)$ of $W+_b c$ are the same word. 

Define the \emph{sign} of the quadruple to be $1$ if $b \leq a + \ell - 2$, and $-1$ if $b \geq a + \ell - 1$. The sign tells us if the alphabet we inserted at gap $b$ lies in the factor $(W+_b c)[a, a + \ell)$ or $(W+_b c)[a + \ell, a + 2\ell)$.

% We say $(a, \ell, b, c)$ is \emph{trivially square-completing in $W$} if the letter $c$ is the same as letter $b$ or $b + 1$. We say it is \emph{nontrivially square-completing in $W$} if otherwise.
\end{definition}

\begin{proposition}\label{key prop} Let $W$ be a square-free word, and suppose $(a, \ell, b, c)$ and $(a', \ell', b', c')$ are square-completing quadruples in $W$ with the same sign. Then one of the following holds: \begin{enumerate}
\item one of $a,b,\ell$ differs by at least $\frac{1}{5} L - 2$  from the corresponding $a', b', \ell'$, where $L = \max(\ell, \ell')$;
\item $b = b'$ and $c=c'$.
\end{enumerate}
Furthermore, if $\ell = \ell'$, then one of the following holds
\begin{enumerate}
\item $\abs{a - a'} \geq L - 1$;
\item $b = b'$ and $c=c'$.
\end{enumerate}
\end{proposition}
\begin{proof}
Suppose the contrary, and neither (1) or (2) are satisfied. Then we have $\ell, \ell' \in  [4L / 5, L]$. By symmetry, we can assume without loss of generality that $\ell \leq \ell'$, and the sign of both quadruples is $1$, that is $b \leq a + \ell - 2$ and $b' \leq a' + \ell' - 2$. 

We argue by cases. In Case 1 we assume $\ell' - \ell \geq 1$, and in Case 2 we handle the scenario $\ell' = \ell = L$. Each case is split into two subcases.

{\bf Case 1.1.} $M = \max(b, b') \leq a + \frac{3L}{5}.$ Then, consider the word $W[M + 1, M + 1 + \ell' - \ell)$. We know that the factor $(W +_b c)[a, a + \ell)$ and the factor $(W +_b c)[a + \ell, a + 2\ell)$ of $W$ are the same word. As $M + 1 > b$, we have
$$W[M + 1, M + 1 + \ell' - \ell) = (W+_b c)[M + 2, M + 2 + \ell' - \ell)$$
On the other hand, we have
$$M + 2 + \ell' - \ell \leq a + \frac{3L}{5} + 2 + \frac{L}{5} \leq a + \ell.$$
Therefore, $(W+_b c)[M + 2, M + 2 + \ell' - \ell)$ is a factor of $(W +_b c)[a, a + \ell)$, so it is equal to the corresponding factor of $(W +_b c)[a + \ell, a + 2\ell)$. More precisely,
$$(W+_b c)[M + 2, M + 2 + \ell' - \ell) = (W+_b c)[M + 2 + \ell, M + 2 + \ell').$$
Thus we have 
$$W[M + 1, M + 1 + \ell' - \ell) = W[M + 1 + \ell, M + 1 + \ell').$$
Similarly, since
$$a' < M + 1, M + 2 + \ell' - \ell \leq a + \frac{4L}{5} + 2\leq a' + L = a' + \ell',$$
we have $(W+_b c)[M + 2, M + 2 + \ell' - \ell)$ is a factor of $(W +_b c)[a', a' + \ell')$, so we conclude that
$$(W+_b c)[M + 2, M + 2 + \ell' - \ell) = (W+_b c)[M + 2 + \ell', M + 2 + 2\ell' - \ell).$$
Thus we have
$$W[M + 1, M + 1 + \ell' - \ell) = W[M + 1 + \ell', M + 1 + 2\ell' - \ell).$$
But then we have
$$W[M + 1 + \ell, M + 1 + \ell') = W[M + 1 + \ell', M + 1 + 2\ell' - \ell),$$
and we have found a square in $W$, which is a contradiction.

{\bf Case 1.2.} $M = \max(b, b') > a + \frac{3L}{5}$. In this case, as $\abs{b - b'} \leq \frac{L}{5} - 2$, we have $\min(b, b') > a + \frac{2L}{5} + 2$, and therefore $\min(b, b') > \max(a,a') + \frac{L}{5} + 4$. Let $A = \max(a, a')$. Then we note that
$$A + \ell' - \ell \leq A + \frac{L}{5}-2 < \min(b,b').$$
So we conclude that
$$W[A, A + \ell' - \ell) = (W+_b c)[A, A + \ell' - \ell)$$
and
$$W[A, A + \ell' - \ell) = (W+_{b'} c')[A, A + \ell' - \ell).$$
As $\min(b, b') \leq b < a + \ell$, we have $(W+_b c)[A, A + \ell' - \ell)$ is a factor of $(W+_b c)[a, a + \ell)$. So we conclude that
$$(W+_b c)[A, A + \ell' - \ell) = (W+_b c)[A + \ell, A + \ell') = W[A + \ell - 1, A + \ell' - 1).$$
Similarly, because $(W+_{b'} c')[A, A + \ell' - \ell)$ is a factor of $(W+_{b'} c')[a, a' + \ell')$, and
$$A + \ell' - 1 \geq a' + \ell' - 1 \geq b' + 1,$$
we conclude that
$$(W+_{b'} c')[A, A + \ell' - \ell) = (W+_{b'} c')[A + \ell', A + 2\ell' - \ell) = W[A + \ell' - 1, A + 2\ell' -\ell - 1).$$
Then
$$W[A + \ell - 1, A + \ell' - 1) = W[A + \ell' - 1, A + 2\ell' -\ell - 1)$$
and we have found a square in $W$, which is a contradiction.

{\bf Case 2.1.} $b\neq b'$. Without loss of generality, let $b<b'.$ Recall our assumption that $|a-a'|<L-1$. We do additional case work on whether $b' \geq a'$ or $b' = a' - 1$, i.e. whether the inserted character at $b'$ is at the start of the square in $W+_{b'} c'$ or not.

First we handle the case $b' \geq a'$. We first note that it is impossible for $b = a + L - 2$, as if so then we have
$$W[a, a + L - 1) = (W+_b c)[a, a + L - 1) = (W+_b c)[a + L, a + 2L - 1) = W[a + L - 1, a _ 2L - 2)$$
and we have found a square in $W$, which is a contradiction. Hence, we have 
$$b \leq a + L - 3.$$ Furthermore, by assumption we have
$$a' \leq a + L - 2.$$
Therefore, if we let $i = \max(a',b + 1)$, then by assumption we have $i + 1 \leq a + L - 1$, so
$$W[i, i + 1) = (W+_b c)[i + 1, i + 2) = (W+_bc)[i+1+L, i+2+L)= W[i+L, i+1+L).$$
On the other hand, as $i \leq b'$, we have
$$W[i, i + 1) = (W+_{b'} c')[i, i + 1) = (W+_{b'} c')[i+L, i+1+L)= W[i-1+L, i+L).$$
Thus we conclude that
$$W[i+L, i+1+L) = W[i-1+L, i+L).$$
So we have found a square in $L$, which is a contradiction.

Then we handle the case $b' = a' - 1$. In this case, we have
$$c' = (W+_{b'} c')[a', a' + 1) = (W+_{b'} c')[a' + L, a' + L + 1) = W[a' + L - 1, a' + L).$$
Note that as $b' > b$, we have $a' + L - 1 > b$, so
$$c ' =W[a' + L - 1, a' + L) = (W+_b c)[a' + L, a' + L + 1).$$
As $a' + L > b + L + 1 \geq a + L$, and $a' + L \leq a + 2L - 1$, we find that $(W+_b c)[a' + L, a' + L + 1)$ is a letter in $(W+_b c)[a + L, a + 2L - 1)$. Therefore,
$$c' = (W+_b c)[a' + L, a' + L + 1) = (W+_b c)[a', a' + 1).$$
Since $a' = b' + 1\geq b + 2$, we get
$$c' = (W+_b c)[a', a' + 1) = W[a' - 1, a').$$
However, this implies that
$$W[a' - 1, a' + L - 1) = (W+_{b'}c')[a', a' + L) = (W+_{b'}c')[a' + L, a' + 2L) = W[a' + L - 1, a' + 2L - 1),$$
so we have found a square in $W$, which is a contradiction.

{\bf Case 2.2.} $b=b'$. We know $(W+_bc)[b+1,b+2)=c$ and $(W+_bc)[a,a+L)=(W+_bc)[a+L,a+2L)$, so $(W+_bc)[b+1+L,b+2+L)=c$. This implies $$W[b+L,b+1+L)=c.$$ The exact same logic also gives $W[b+L,b+1+L)=c'$. So $c = c'$, which is a contradiction.
\end{proof}
\begin{corollary} 
\label{cor:counting}
Let $W$ be an $n$-letter square-free word $W$. Let $\cA$ be a set of square-completing quadruples $(a, \ell, b, c)$ such that no two elements of $\cA$ share the same $(b, c)$. For each $L \geq 2$, define
$$\cA_L = \cA \cap \{(a, L, b, c): a,b \in \ZZ, c \text{ any letter}\}.$$
Then for any $L \geq 2$, we have
$$|\cA_L| \leq \frac{2n}{L - 1}.$$
Furthermore, for any $L \geq 300$, we have 
$$\sum_{\ell = L}^{2L - 1}|\cA_\ell| \leq \frac{320n}{L}.$$
\end{corollary}
\begin{proof}
To prove the first proposition for $L \geq 2$, note that for a given sign $\epsilon \in \{-1, 1\}$, and any two quadruples $(a, L, b, c)$ and  $(a', L, b', c')$ in $\cA_L$ with sign $\epsilon$, by \cref{key prop} we must have $\abs{a - a'} \geq L - 1$. Thus over all quadruples in $\cA_L$ with sign $\epsilon$, the $a$'s must be spaced at least $L$ apart, and must be in the range $[1, n - 2L + 2]$. Therefore, there are at most
$$\frac{n - 2L + 2}{L - 1} + 1 \leq \frac{n}{L - 1}$$
such quadruples. Given there are two possible signs, the total number of quadruples is at most $\frac{2n}{L - 1}$.

To prove the second statement, we can divide the range $[1, n - 2L + 2]$ into at most
$$\frac{n - 2L + 2}{L / 6} + 1 \leq \frac{6n}{L}$$
intervals of length $\frac{L}{6}$. For each such interval $I = [x,x + \frac{L}{6})$, define
$$\cB_I = \{(a, \ell, b, c): (a, \ell, b, c) \in \cA, \ell \in [L, 7L / 6), a \in I, (a, \ell, b, c) \text{ has sign 1}\}.$$
Assume $(a, \ell, b, c)$ and  $(a', \ell', b', c')$ are two distinct quadruples in $\cB_I$. Note that
$$\abs{\ell - \ell'} \leq \frac{L}{6} < \frac{L}{5} - 2,$$
and
$$\abs{a - a'}\leq \frac{L}{6} < \frac{L}{5} - 2.$$
Thus by \cref{key prop}, we must have $b, b'$ spaced at least $L / 5 - 2$ apart. Furthermore, for each quadruple $(a,\ell,b,c)$ in $\cB_I$, the $b$'s are restricted to the interval $[x - 1, x + \floor{\frac{4L}{3}})$ due to the quadruples having sign $1$. Thus the size of $\cB_I$ is upper bounded by
$$\frac{\floor{\frac{4L}{3}} + 1}{\frac{L}{5} - 2} + 1$$
For $L \geq 300$, we can verify that
$$\frac{\floor{\frac{4L}{3}} + 1}{\frac{L}{5} - 2} + 1 < 8$$
which implies
$$\abs{\cB_I} \leq 7.$$
Symmetrically, if we let
$$\cC_I = \{(a, \ell, b, c): (a, \ell, b, c) \in \cA, \ell \in [L, 7L / 6), a \in I, (a, \ell, b, c) \text{ has sign -1}\}$$
then
$$\abs{\cC_I} \leq 7.$$
Summing over all the intervals , we conclude that
$$\sum_{L \leq \ell < 7L / 6} | \cA_{\ell}| = \sum_I (\abs{\cB_I} + \abs{\cC_I}) \leq 14 \cdot \frac{6n}{L}.$$
Analogously, we have for any non-negative integer $i$.
$$\sum_{(7 / 6)^i L \leq \ell < (7 / 6)^{i + 1} L} |\cA_{\ell}| \leq \frac{7 \cdot 2 \cdot 6n}{(7 / 6)^iL}.$$
Summing over $i \in \{0,1,2,3,4\}$, we obtain that 
$$\sum_{L \leq \ell < 2L} |\cA_{\ell}| \leq \frac{320 n}{L}.$$
as desired.
\end{proof}
\begin{proof}[Proof of \cref{main-words}]
Let $W$ be any extremal square-free word on an alphabet of size $k$. Then for any gap $0 \leq b \leq n + 1$ and any letter $c$ not equal to the two letters adjacent to gap $b$, there exists some $a$ and $\ell \geq 2$ such that $(a, \ell, b, c)$ is a square-completing quadruple in $W$. Let $\cA$ be the set consisting of one such quadruple for each choice of $(b, c)$. On one hand, by construction we have
$$|\cA| \geq (k - 2)(n + 1).$$
On the other hand, by \cref{cor:counting}, we have, 
\begin{align*}
 | \cA| &= \sum_{\ell = 2}^\infty |\cA_\ell| \\
 &= \sum_{\ell = 2}^{319} |\cA_\ell| + \sum_{j = 0}^{\infty}\ \sum_{\ell \in [320\cdot 2^j, 320\cdot 2^{j + 1})} |\cA_\ell| \\
 &\leq \sum_{\ell = 2}^{319} \frac{2n}{\ell - 1} + \sum_{j = 0}^{\infty} \frac{320n}{320\cdot 2^j}  < 14.7n.
\end{align*}
Thus we conclude that $k < 17$, as desired.
\end{proof}

\begin{proof}[Proof of \cref{cor-1415}]
We have shown that the number of square-completing quadruples $(a, \ell, b, c)$ such that no two elements share the same $(b, c)$ in a square-free word $W$ of length $n$ is less than $14.7n$. Thus, the number of ways to insert an alphabet into $W$ such that the result is no longer square-free is less than $16.7n$. Therefore, if the alphabet size is at least $17$, then it is possible to insert a letter into the latter $\frac{16.7}{17}$ of any square-free word $W$ such that the result is square-free. So the prefix of $G_i$ will stabilize and the sequence converges to an infinite limit word.
\end{proof}
% %%%
% \section{Extremal abelian square-free words}
% \begin{conjecture}\label{abelian-conj}
% There are infinitely many extremal abelian square-free words over a 4-letter alphabet.
% \end{conjecture}
% \begin{observation}
% If \cref{abelian-conj} is true, then there are infinitely many extremal abelian square-free words on any finite alphabet of size at least $4$.
% \end{observation}
% \begin{proof}
% Suppose $X$ is an extremal abelian square-free word on a size-$k$ alphabet $\{u_0,u_1,\ldots,u_{k-1}\}$, and $u_{k}$ is a new letter. Let $X^{\text{rev}}$ be the reversal of $X$. Then $Xu_{k}X^{\text{rev}}$ is also an extremal abelian square-free word on the size-$(k+1)$ alphabet. We conclude by recursively applying the construction.
% \end{proof}
% %%%


\begin{thebibliography}{99}
\bibitem{Berstel} J. Berstel, \emph{Axel Thue’s papers on repetitions in words: a translation.} Publications du LaCIM, Vol. 20, Universit\'e du Qu\'ebec a Montr\'eal, 1995.

\bibitem{BP} J. Berstel and D. Perrin. \emph{The origins of combinatorics on words}. European J. Combin., 28 (2007), 996--1022. \url{https://doi.org/10.1016/j.ejc.2005.07.019}

\bibitem{EJS} R. C. Entringer, D. E. Jackson, and J. A. Schatz. \emph{On nonrepetitive sequences}. J. Combin. Theory Ser. A., 16, 1974, 159--164. \url{https://doi.org/10.1016/0097-3165(74)90041-7}

\bibitem{Erdos} P. Erd\H{o}s. \emph{Some unsolved problems.} Magyar Tud. Akad. Mat. Kutat\'o Int. K\"ozl., 6, 1961, 221--254. \url{https://doi.org/10.1307/MMJ/1028997963}

\bibitem{GKN} J. Grytczuk, H. Kordulewski, and A. Niewiadomski. \emph{Extremal square-free words}. Electron. J. Combin., 27(1) (2020), P1.48. \url{https://doi.org/10.37236/9264}

\bibitem{GKP} J. Grytczuk, H. Kordulewski, and B. Pawlik. \emph{Square-free extensions of words}. arXiv preprint. \url{https://arxiv.org/abs/2104.04841v3}

%\bibitem{Keranen} V. Ker\"anen. \emph{Abelian squares are avoidable on 4 letters}. International Colloquium on Automata, Languages, and Programming, (1992), 41--52. \url{https://doi.org/10.1007/3-540-55719-9_62}

\bibitem{MR} L. Mol and N. Rampersad. \emph{Lengths of extremal square-free ternary words}. Contrib. Math., Vol. 16 No. 1 (2021). \url{https://doi.org/10.11575/cdm.v16i1.69831}

%\bibitem{Pleasants} P. A. Pleasants. \emph{Non-repetitive sequences}. Proc. Cambridge Phil, Soc., 68, 1970, 267--274.

\bibitem{TZ} N. Ter-Saakov and E. Zhang. \emph{Extremal Pattern-Avoiding Words}. arXiv preprint. \url{https://arxiv.org/abs/2009.10186}

\bibitem{Thue} A. Thue. \emph{\"Uber unendliche zeichenreihen}. Norske vid. Selsk. Skr. Mat. Nat. Kl., 7 (1906), 1--22.

\bibitem{Thue2} A. Thue. \emph{\"Uber die gegenseitige Lage gleicher Teile gewisser Zeichenreihen.} Norske vid. Selsk. Skr. Mat. Nat. Kl., 1 (1912), 1--67.
\end{thebibliography}
\end{document}